\documentclass[12pt]{amsart}
\usepackage{bbding}
\usepackage[margin=2.0cm]{geometry}

\usepackage{amsfonts}
\usepackage{mathrsfs}
\usepackage{cite}
\usepackage{tikz}

\newtheorem{theorem}{Theorem}[section]
\newtheorem{lemma}[theorem]{Lemma}
\newtheorem{corollary}[theorem]{Corollary}
\newtheorem{de}[theorem]{Definition}
\newtheorem{example}[theorem]{Example}

\newcommand{\Int}[1]{\text{\rm Int }#1}

\title{On the discrete logarithmic Minkowski problem}
\author{K\'aroly J. B\"or\"oczky}
\address{Alfr\'ed R\'enyi Institute of Mathematics, Hungarian Academy
  of Sciences, Reltanoda u. 13-15, H-1053 Budapest, Hungary, and
Department of Mathematics, Central European University, Nador u 9, H-1051, Budapest, Hungary,boroczky.karoly.j@renyi.mta.hu}
\author{P\'al Heged\H{u}s}
\address{Department of Mathematics, Central European University, Nador u 9, H-1051, Budapest, Hungary,hegedusp@ceu.edu}
\author{Guangxian Zhu}
\address{Department of Mathematics, Polytechnic School of Engineering, New York University,  Six Metrotech, Brooklin, NY 11201, USA,guangxian.zhu@gmail.com}
\thanks{2010 \emph{Mathematics Subject Classification}: 52A40.\\
\emph{Key Words}: Polytope, cone-volume measure, Minkowski problem,
$L_{p}$ Minkowski problem, logarithmic Minkowski problem, Monge-Amp\`{e}re equation.}

\begin{document}
\maketitle

\begin{abstract}
A new sufficient condition for the existence of a solution for the logarithmic Minkowski problem is established. This new condition contains the one established by Zhu\cite{Z2} and the discrete case established by B\"or\"oczky, Lutwak, Yang, Zhang \cite{BLYZ} as two important special cases.
\end{abstract}

\section{Introduction}

The setting for this paper is $n$-dimensional Euclidean space $\mathbb{R}^{n}$. A \emph{convex body} in $\mathbb{R}^{n}$ is a compact convex set that
has non-empty interior. If $K$ is a convex body in $\mathbb{R}^{n}$, then the
\emph{surface area measure}, $S_{K}$, of $K$ is a Borel measure on the unit sphere, $S^{n-1}$, defined for a Borel $\omega\subset S^{n-1}$  (see, e.g., Schneider \cite{SCH}), by
$$
S_{K}(\omega)=\int_{x\in\nu_{K}^{-1}(\omega)}d\mathcal{H}^{n-1}(x),
$$
where $\nu_{K}:\partial'K\rightarrow S^{n-1}$ is the Gauss map of $K$, defined on $\partial'K$, the set of points of $\partial K$ that have a
unique outer unit normal, and $\mathcal{H}^{n-1}$ is ($n-1$)-dimensional Hausdorff measure.

As one of the cornerstones of the classical Brunn-Minkowski theory, the Minkowski's existence theorem can be stated as follows  (see, e.g., Schneider \cite{SCH}): If $\mu$ is not concentrated on a great subsphere of $S^{n-1}$, then $\mu$ is the surface area measure of a convex body if and only if
$$
\int_{S^{n-1}}ud\mu(u)=0.
$$
The solution is unique up to translation, and even the regularity of the solution is well investigated,  see e.g.,
Lewy \cite{LE}, Nirenberg \cite{NIR}, Cheng and Yau \cite{CY}, Pogorelov \cite{POG}, and Caffarelli \cite{LC}.

The surface area measure of a convex body has clear geometric
significance. Another important measure that is associated with a convex body and that
has clear
geometric importance is the cone-volume measure. If $K$ is a convex
body in $\mathbb{R}^{n}$ that contains the origin in its interior,
then
the \emph{cone-volume measure}, $V_{K}$, of $K$ is a Borel measure on
$S^{n-1}$ defined for each Borel $\omega\subset S^{n-1}$ by
$$
V_{K}(\omega)=\frac{1}{n}\int_{x\in\nu_{K}^{-1}(\omega)}x\cdot\nu_{K}(x)\, d\mathcal{H}^{n-1}(x).
$$
For references regarding cone-volume measure see, e.g., \cite{BH, BLYZ, BLYZ2, BLYZ3, LU1, LU2, LR, NAO, NR, PAO, ST1, ST2, ST3, Z2}.

The Minkowski's existence theorem deals with the question of prescribing the
surface area measure. The following problem is prescribing the
cone-volume measure.\\ 

\textbf{Logarithmic Minkowski problem:} What are the necessary and
sufficient conditions on a finite Borel measure $\mu$ on $S^{n-1}$ so
that $\mu$ is the cone-volume measure of a convex body in
$\mathbb{R}^{n}$?\\

In \cite{LUT},  Lutwak showed that there is an $L_{p}$ analogue of the surface area measure (known as the $L_{p}$ surface area measure). In recent years, the $L_{p}$ surface area measure appeared in, e.g., \cite{ADA, BG, CG, GM, H1, HP1, HP2, HENK, LU1, LU2, LR, LYZ1, LYZ2, LYZ3, LYZ6, LZ, NAO, NR, PAO, PW, ST3}. In \cite{LUT}, Lutwak posed the associated $L_{p}$ Minkowski problem which extends the classical Minkowski problem for $p\geq 1$.
In addition, the $L_p$ Minkowski problem for $p<1$ was publicized by a series of talks by Erwin Lutwak in the 1990's. 
 The $L_{p}$ Minkowski problem is the classical Minkowski problem when $p=1$, while the $L_{p}$ Minkowski problem is the logarithmic Minkowski problem when $p=0$. The $L_{p}$ Minkowski problem is interesting for all real $p$, and have been studied by, e.g., Lutwak \cite{LUT}, Lutwak and Oliker \cite{LO}, Chou and Wang \cite{CW},  Guan and Lin \cite{GL}, Hug, et al. \cite{HLYZ2}, B\"{o}r\"{o}czky, et al. \cite{BLYZ}. Additional references regarding the $L_{p}$ Minkowski problem and Minkowski-type problems can be found in, e.g., \cite{BLYZ, WC, CW, GG, GL, GM, H1, HL1, HLYZ2, HLYZ1, HMS, JI, KL, LWA, LUT, LO, LYZ5, MIN, ST1, ST2, Z3, Z4}. Applications of the solutions to the $L_{p}$ Minkowski problem can be found in, e.g., \cite{AN1, AN2, KSC, Z, GH, LYZ4, CLYZ, HS1, HUS, Iva13, HS2, HSX, TUO}.

A finite Borel measure $\mu$ on $S^{n-1}$ is said to satisfy the
\emph{subspace concentration condition} if, for every subspace $\xi$
of $\mathbb{R}^{n}$, such that $0<\dim\xi<n$,
\begin{equation}\tag{1.2}\label{Equation 1.2}
\mu(\xi\cap S^{n-1})\leq\frac{\dim\xi}{n}\mu(S^{n-1}),
\end{equation}
and if equality holds in (\ref{Equation 1.2}) for some subspace $\xi$, then there exists a
subspace $\xi'$, that is complementary to $\xi$ in $\mathbb{R}^{n}$,
so that also
$$
\mu(\xi'\cap S^{n-1})=\frac{\dim\xi'}{n}\mu(S^{n-1}).
$$

The measure $\mu$ on $S^{n-1}$ is said to satisfy the \emph{strict subspace concentration inequality} if the inequality in (\ref{Equation 1.2}) is strict for each subspace $\xi\subset\mathbb{R}^{n}$, such that $0<\dim\xi<n$.

Very recently, B\"{o}r\"{o}czky and Henk \cite{BH}  proved that if the centroid of a convex body is the origin, then the cone-volume measure of this convex body satisfies the subspace concentration condition. For more references on the progress of the subspace concentration condition, see, e.g., Henk et al. \cite{HSW}, He et al. \cite{HLK}, Xiong \cite{XIO}, B\"{o}r\"{o}czky et al. \cite{BLYZ3}, and Henk and Linke \cite{HENK}.

In \cite{BLYZ}, B\"{o}r\"{o}czky, et al. established the
following necessary and sufficient conditions for the existence of
solutions to the
even logarithmic Minkowski problem.

\begin{theorem}[B\"or\"oczky,Lutwak,Yang,Zhang]
\label{theoA}
A non-zero finite even Borel measure on $S^{n-1}$ is the cone-volume measure of an
origin-symmetric convex body in $\mathbb{R}^{n}$ if and only if it
satisfies the subspace concentration condition.
\end{theorem}

The convex hull of a finite set is called a polytope provided that it has positive $n$-dimensional volume.
The convex hull of a subset of these points is called a \emph{facet} of the polytope if it lies entirely on the boundary of the polytope and has
positive ($n-1$)-dimensional volume. If a polytope $P$ contains the origin in its interior and has $N$ facets whose outer unit normals are
$u_{1},...,u_{N}$, and such that if the facet with outer unit normal $u_{k}$ has $(n-1)$-measure $a_{k}$ and distance from the origin $h_{k}$ for
all $k\in\{1,...,N\}$, then
$$
V_P=\frac1n\sum_{k=1}^{N}h_{k}a_{k}\delta_{u_{k}}.
$$
where $\delta_{u_{k}}$ denotes the delta measure that is concentrated at the point $u_{k}$.

A finite subset $U$ (with no less than $n$ elements) of $S^{n-1}$ is said to be \emph{in general position} if any $k$ elements of $U$, $1\le k\le n$,  are linearly independent.

For a long time, people believed that the data for a cone-volume measure can not be arbitrary. However, Zhu \cite{Z2} proved that any discrete measure on $S^{n-1}$ whose support is in general position is a cone-volume measure.

\begin{theorem}[Zhu]
\label{theoB}
A discrete measure, $\mu$, on the unit sphere $S^{n-1}$ is the cone-volume measure of a polytope whose outer unit normals are in general position if and only if the support of $\mu$ is in general position and not concentrated on a closed hemisphere of $S^{n-1}$.
\end{theorem}

A linear subspace $\xi$ ($1\leq\dim\xi\leq n-1$) of  $\mathbb{R}^{n}$ is said to be essential with respect to a Borel measure $\mu$ on $S^{n-1}$ if $\xi\cap{\rm supp}(\mu)$ is not concentrated on any closed hemisphere of $\xi\cap S^{n-1}$. 

\begin{de} 
\label{essentialdef}
A finite Borel measure $\mu$ on $S^{n-1}$ is said to satisfy the
essential subspace concentration condition if, for every essential subspace $\xi$ (with respect to $\mu$)
of $\mathbb{R}^{n}$, such that $0<\dim\xi<n$,
\begin{equation}\tag{1.3}\label{Equation 1.3}
\mu(\xi\cap S^{n-1})\leq\frac{\dim\xi}{n}\mu(S^{n-1}),
\end{equation}
and if equality holds in (\ref{Equation 1.3}) for some essential subspace $\xi$ (with respect to $\mu$), then there exists a
subspace $\xi'$, that is complementary to $\xi$ in $\mathbb{R}^{n}$,
so that
\begin{equation}\tag{1.4}
\mu(\xi'\cap S^{n-1})=\frac{\dim\xi'}{n}\mu(S^{n-1}).
\end{equation}
\end{de}

\begin{de} 
The measure $\mu$ on $S^{n-1}$ is said to satisfy the strict essential subspace concentration inequality if the inequality in (\ref{Equation 1.3}) is strict for each essential subspace $\xi$ (with respect to $\mu$) of $\mathbb{R}^{n}$, such that $0<\dim\xi<n$.
\end{de}
We would like to note that if $\mu$ is a Borel measure on the unit sphere that is not concentrated on a closed hemisphere and satisfies the essential subspace concentration condition, and $\xi$ is an essential subspace (with respect to $\mu$) that reaches the equality in (1.3), then by Lemma~\ref{Lemma 5.2}, $\xi'$ (in (1.4)) is an essential subspace with respect to $\mu$.

It is the aim of this paper to establish the following.

\begin{theorem}
\label{sccgen}
If $\mu$ is a discrete measure on $S^{n-1}$ that is not concentrated on any closed hemisphere and satisfies the essential subspace concentration condition, then $\mu$ is the cone-volume measure of a polytope in $\mathbb{R}^{n}$ containing the origin in its interior.
\end{theorem}

If $\mu$ is a non-trivial even Borel measure on $S^{n-1}$, and $\xi$ is a $k$-dimensional linear subspace  of $\mathbb{R}^{n}$ spanned by some vectors $v_1,\ldots,v_k\in{\rm supp}(\mu)$ for $1\leq k\leq n-1$, then $-v_1,\ldots,-v_k\in{\rm supp}(\mu)$, as well, and hence $\xi$
 is an essential subspace. In particular,  for even discrete measures, Theorem~\ref{sccgen} is equivalent to the sufficient condition of Theorem~\ref{theoA}. However, there are non-even discrete measures that satisfy the 
 essential subspace concentration condition, but not the subspace concentration condition. For example, if a $k$-dimensional subspace $\xi$, $1\leq k\leq n-1$,  intersects the support of the measure in $k+1$ unit vectors $u_0,\ldots,u_k$ such that $u_1,\ldots,u_k$ are independent, and $u_0=\alpha_1u_1+\ldots+\alpha_ku_k$ for $\alpha_1,\ldots,\alpha_k>0$, then there is no condition on the restriction of the measure to $\xi\cap S^{n-1}$. 
Therefore, for discrete measures, Theorem~\ref{sccgen} is a generalization of the sufficient condition of Theorem~\ref{theoA}.

We claim that if the support of a discrete measure $\mu$ is in general position, then the set of essential subspaces (with respect to $\mu$) is empty. Otherwise, there exists a subspace $\xi$ with $1\leq\dim\xi\leq n-1$ such that supp$(\mu)\cap\xi$ is not concentrated on a closed hemisphere of $S^{n-1}\cap\xi$. Then we can choose $\dim\xi+1$ $(\leq n)$ vectors from supp$(\mu)\cap\xi$ that are linearly dependent. But this contradicts the fact that supp$(\mu)$ is in general position. From our declaration, we have, Theorem~\ref{sccgen} contains Theorem~\ref{theoB} as an important special case.

In $\mathbb{R}^{2}$, Theorem~\ref{sccgen} leads to the main result of Stancu (\cite{ST1}, pp. 162), where she applied a different method called the crystalline deformation.

New inequalities for cone-volume measures are established in section~\ref{secnec}.

\section{Preliminaries}

In this section, we collect some basic notations and facts about convex bodies. For general references regarding convex bodies see, e.g., \cite{GAR1,PM, GRU, SCH, THO}.

The vectors of this paper are column vectors. For $x, y\in\mathbb{R}^{n}$, we will write $x\cdot y$ for the standard
inner product of $x$ and $y$, and write $|x|$ for the Euclidean norm of $x$. We write $S^{n-1}=\{x\in\mathbb{R}^{n}: |x|=1\}$ for
the boundary of the Euclidean unit ball $B^{n}$ in $\mathbb{R}^{n}$, and write $\kappa_{n}$ for the volume of the unit ball.
 Let  $V_k(M)$ denote the $k$-dimensional Hausdorff measure of an at most $k$-dimensional convex set $M$. In addition, if $k=n-1$, then we also use the notation $|M|$.

Suppose $X_{1}, X_{2}$ are subspaces of $\mathbb{R}^{n}$, we write $X_{1}\perp X_{2}$ if $x_{1}\cdot x_{2}=0$ for all $x\in X_{1}$ and $x_{2}\in X_{2}$. Suppose $X$ is a subspace of $\mathbb{R}^{n}$ and $S$ is a subset of $\mathbb{R}^{n}$, we write $S|_{X}$ for the orthogonal projection of $S$ on $X$.

Suppose $C$ is a subset of $\mathbb{R}^{n}$, the positive hull, $\textmd{pos}(C)$, of $C$ is the set of all positive combinations of any finitely many elements of $C$. Let $\textmd{lin}(C)$ be the smallest linear subspace of $\mathbb{R}^{n}$ containing $C$. The diameter of $C$ is defined by
$$
d(C)=\sup\{|x-y|: x,y\in C\}.
$$

For $K_{1}, K_{2}\subset\mathbb{R}^{n}$ and $c_{1}, c_{2}\geq0$, the Minkowski combination, $c_{1}K_{1}+c_{2}K_{2}$, is defined by
$$
c_{1}K_{1}+c_{2}K_{2}=\{c_{1}x_{1}+c_{2}x_{2}: x_{1}\in K_{1}, x_{2}\in K_{2}\}.
$$
The \emph{support function}
$h_{K}: \mathbb{R}^{n}\rightarrow\mathbb{R}$ of a compact convex set
$K$ is defined, for $x\in\mathbb{R}^{n}$, by
$$
h(K,x)=\max\{x\cdot y: y\in K\}.
$$
Obviously, for $c\geq0$ and $x\in\mathbb{R}^{n}$, we have
$$
h(cK,x)=h(K,cx)=ch(K,x).
$$

The \emph{convex hull} of two convex sets $K, L$ in
$\mathbb{R}^{n}$ is defined by
$$
[K,L]=\{z: z=\lambda x+(1-\lambda)y, 0\leq\lambda\leq1\textmd{ and }x,y\in K\cup L\}.
$$
The \emph{Hausdorff distance} of two compact sets $K, L$ in
$\mathbb{R}^{n}$ is defined by
$$
\delta(K,L)=\inf\{t\geq0: K\subset L+tB^{n}, L\subset K+tB^{n}\}.
$$
It is known that the Hausdorff distance between two convex bodies, $K$ and $L$, is
$$
\delta(K,L)=\max_{u\in S^{n-1}}|h(K,u)-h(L,u)|.
$$
We always consider the space of convex bodies as metric space equipped with the Hausdorff distance. It is known that if a sequence $\{K_m\}$ of convex bodies tends to a convex body $K$ in $\mathbb{R}^{n}$ containing the origin in its interior, then $S_{K_m}$ tends weakly to $S_K$, and hence $V_{K_m}$ tends weakly to $V_K$ (see Schneider \cite{SCH}).

For a convex body $K$ in $\mathbb{R}^{n}$, and $u\in S^{n-1}$,
the \emph{support hyperplane} $H(K,u)$ in direction $u$ is defined by
$$
H(K,u)=\{x\in\mathbb{R}^{n}:x\cdot u=h(K,u)\},
$$
the \emph{face} $F(K,u)$ in direction $u$ is defined by
$$
F(K,u)=K\cap H(K,u).
$$

Let $\mathcal{P}$ be the set of all polytopes in $\mathbb{R}^{n}$. If the unit vectors $u_{1},...,u_{N}$ are not concentrated on a closed hemisphere, let $\mathcal{P}(u_{1},...,u_{N})$ be the set of all polytopes $P\in\mathcal{P}$ such that the set of outer unit normals of the facets of $P$ is a subset of $\{u_{1},...,u_{N}\}$, and let $\mathcal{P}_{N}(u_{1},...,u_{N})$ be the the set of all polytopes $P\in\mathcal{P}$ such that the set of outer unit normals of the facets of $P$ is $\{u_{1},...,u_{N}\}$.

\section{An extremal problem related to the logarithmic Minkowski problem}

Let us suppose $\gamma_{1},...,\gamma_{N}\in(0,\infty)$, and the unit vectors $u_{1},...,u_{N}$ are not concentrated on a closed hemisphere. Let
\begin{equation}\tag{3.0}\label{Equation 3.0}
\mu=\sum_{i=1}^{N}\gamma_{i}\delta_{u_{i}},
\end{equation}
and for $P\in\mathcal{P}(u_{1},...,u_{N})$ define $\Phi_{P}:$
$\Int(P)\rightarrow\mathbb{R}$ by
\begin{equation}\tag{3.1}\label{Equation 3.1}
\begin{split}
\Phi_{P}(\xi)&=\int_{S^{n-1}}\log\left(h(P,u)-\xi\cdot u\right)d\mu(u)\\
&=\sum_{k=1}^{N}\gamma_{k}\log\big(h(P,u_{k})-\xi\cdot u_{k}\big),
\end{split}
\end{equation}
where $\Int(P)$ is the interior of $P$.

In this section, we study the following extremal problem:
\begin{equation}\tag{3.2}\label{Equation 3.2}
\inf\left\{\max_{\xi\in\Int (Q)}\Phi_{Q}(\xi):
Q\in\mathcal{P}(u_{1},...,u_{N})\textmd{ and
}V(Q)=|\mu|\right\},
\end{equation}
where $|\mu|=\sum_{k=1}^{N}\gamma_{k}$.

We will prove that the solution of problem (\ref{Equation 3.2}) solves the corresponding logarithmic Minkowski problem.

For the case where $u_{1},...,u_{N}$ are in general position and $Q\in\mathcal{P}_{N}(u_{1},...,u_{N})$, problem (\ref{Equation 3.2}) was studied in \cite{Z2}. The results and proofs in this section are similar to \cite{Z2}. However, for convenience of the readers, we give detailed proofs for these results.

\begin{lemma}\label{Lemma 3.1}
Suppose $\mu=\sum_{k=1}^{N}\gamma_{k}\delta_{u_{k}}$ is a discrete measure on $S^{n-1}$ that is not concentrated on a closed hemisphere, and $P\in\mathcal{P}(u_{1},...,u_{N})$, then there exists a unique
point $\xi(P)\in\Int (P)$ such that
$$
\Phi_{P}(\xi(P))=\max_{\xi\in\Int(P)}\Phi_{P}(\xi).
$$
\end{lemma}
\begin{proof}
Let $0<\lambda<1$
and $\xi_{1}, \xi_{2}\in\Int (P)$. From the concavity of the logarithmic function,
\begin{equation*}
\begin{split}
\lambda\Phi_{P}(\xi_{1})+(1-\lambda)\Phi_{P}(\xi_{2})&=\lambda\int_{S^{n-1}}\log\left(h(P,u)-\xi_{1}\cdot u\right)d\mu(u)\\
&\qquad+(1-\lambda)\int_{S^{n-1}}\log\left(h(P,u)-\xi_{2}\cdot u\right)d\mu(u)\\
&=\sum_{k=1}^{N}\gamma_{k}\left[\lambda\log(h(P,u_{k})-\xi_{1}\cdot
u_{k})+(1-\lambda)\log(h(P,u_{k})-\xi_{2}\cdot u_{k})\right]\\
&\leq\sum_{k=1}^{N}\gamma_{k}\log\left[h(P,u_{k})-(\lambda\xi_{1}+(1-\lambda)\xi_{2})\cdot
u_{k}\right]\\
&=\Phi_{P}(\lambda\xi_{1}+(1-\lambda)\xi_{2}),
\end{split}
\end{equation*}
with equality if and only if $\xi_{1}\cdot u_{k}=\xi_{2}\cdot u_{k}$
for all $k=1,...,N$. Since the unit vectors $u_{1},...,u_{N}$ are not concentrated on a closed hemisphere, $\mathbb{R}^{n}$=lin$\{u_{1},...,u_{N}\}$. Thus, $\xi_{1}=\xi_{2}$. Therefore, $\Phi_{P}$ is strictly concave
on $\Int(P)$.

Since $P\in\mathcal{P}(u_{1},...,u_{N})$, for any $x\in\partial P$,
there exists some $i_{0}\in\{1,\ldots,N\}$ such that
$$
h(P,u_{i_{0}})=x\cdot u_{i_{0}}.
$$
Thus, $\Phi_{P}(\xi)\rightarrow-\infty$ whenever $\xi\in\Int(P)$ and $\xi\rightarrow x$. Therefore, there exists a unique interior point
$\xi(P)$ of $P$
such that
$$
\Phi_{P}(\xi(P))=\max_{\xi\in\Int(P)}\Phi_{P}(\xi).
$$
\end{proof}
Obviously, for $\lambda>0$ and $P\in\mathcal{P}(u_{1},...,u_{N})$,
\begin{equation}\tag{3.3}\label{Equation 3.3}
\xi(\lambda P)=\lambda\xi(P),
\end{equation}
and if $P_{i}\in\mathcal{P}(u_{1},...,u_{N})$ and $P_{i}$
converges to a polytope $P$, then
$P\in\mathcal{P}(u_{1},...,u_{N})$.

For the case where $u_{1},...,u_{N}$ are in general position, the following lemma was proved in \cite{Z2}.

\begin{lemma}\label{Lemma 3.2}
Suppose $\mu=\sum_{k=1}^{N}\gamma_{k}\delta_{u_{k}}$ is a discrete measure on $S^{n-1}$ that is not concentrated on a closed hemisphere, $P_{i}\in\mathcal{P}(u_{1},...,u_{N})$ and
$P_{i}$ converges to a polytope $P$, then
$\lim_{i\rightarrow\infty}\xi(P_{i})=\xi(P)$ and
$$
\lim_{i\rightarrow\infty}\Phi_{P_{i}}(\xi(P_{i}))=\Phi_{P}(\xi(P)).
$$
\end{lemma} 
\begin{proof} Since $\xi(P)\in\Int (P)$ by Lemma~\ref{Lemma 3.1}, we have
$$
\liminf_{i\rightarrow\infty}\Phi_{P_{i}}(\xi(P_{i}))\geq \liminf_{i\rightarrow\infty}\Phi_{P_{i}}(\xi(P))=\Phi_{P}(\xi(P)).
$$
Let $z$ be any accumulation point of the sequence $\{\xi(P_{i})\}$; namely, the limit of a subsequence $\{\xi(P_{i'})\}$. Since $\Phi_{P_{i}}(\xi(P_{i}))$ is bounded from below, and $h(P,u_{k})-\xi(P_i)\cdot u_{k}$ is bounded from above 
for $k=1,\ldots,N$, it follows that 
$$
\liminf_{i\rightarrow\infty}(h(P,u_{k})-\xi(P_i)\cdot u_{k})=
\liminf_{i\rightarrow\infty}(h(P_i,u_{k})-\xi(P_i)\cdot u_{k})>0
$$
for $k=1,\ldots,N$, and hence $z\in \Int(P)$. We deduce that
$$
\Phi_{P}(z)=\lim_{i'\rightarrow\infty}\Phi_{P}(\xi(P_{i'}))
=\lim_{i'\rightarrow\infty}\Phi_{P_{i'}}(\xi(P_{i'}))
\geq \liminf_{i\rightarrow\infty}\Phi_{P_i}(\xi(P_i))\geq \Phi_{P}(\xi(P)).
$$
Therefore Lemma~\ref{Lemma 3.1} yields $z=\xi(P)$.
\end{proof}

The following lemma will be needed, as well.

\begin{lemma}\label{Lemma 3.3}
Suppose $\mu=\sum_{k=1}^{N}\gamma_{k}\delta_{u_{k}}$ is a discrete measure on $S^{n-1}$ that is not concentrated on a closed hemisphere, $P\in\mathcal{P}(u_{1},...,u_{N})$, then
$$
\sum_{k=1}^{N}\gamma_{k}\frac{u_{k}}{h(P,u_{k})-\xi(P)\cdot u_k}=0.
$$
\end{lemma}
\begin{proof}
We may assume that $\xi(P)$ is the origin because for $x,\xi\in \Int P$, we have
$\Phi_{P-x}(\xi-x)=\Phi_P(\xi)$.
Since $\Phi_{P}(\xi)$ attains its maximum at the origin that is an interior point of $P$, differentiation gives the desired equation.
\end{proof}

\begin{lemma}\label{Lemma 3.4}
Suppose $\mu=\sum_{k=1}^{N}\gamma_{k}\delta_{u_{k}}$ is a discrete measure on $S^{n-1}$ that is not concentrated on a closed hemisphere,
and there exists a $P\in\mathcal{P}_{N}(u_{1},...,u_{N})$ with
$\xi(P)=0$, $V(P)=|\mu|$ such that
$$
\Phi_{P}(0)=\inf\left\{\max_{\xi\in\Int (Q)}\Phi_{Q}(\xi):
Q\in\mathcal{P}(u_{1},...,u_{N})\textmd{ and
}V(Q)=|\mu|\right\}.
$$
Then,
$$
V_{P}=\sum_{k=1}^{N}\gamma_{k}\delta_{u_{k}}.
$$
\end{lemma}
\begin{proof}
According to Equation (\ref{Equation 3.3}), it is sufficient to
establish the lemma under the assumption that
$|\mu|=1$.

From the conditions, there exists a polytope
$P\in\mathcal{P}_{N}(u_{1},...,u_{N})$ with $\xi(P)$ is the origin and
$V(P)=1$ such that
$$
\Phi_{P}(o)=\inf\left\{\max_{\xi\in\Int (Q)}\Phi_{Q}(\xi):
Q\in\mathcal{P}(u_{1},...,u_{N})\textmd{ and }V(Q)=1\right\}.
$$

For $\tau_{1},...,\tau_{N}\in\mathbb{R}$, choose $|t|$ small
enough so that the polytope
$$
P_{t}=\bigcap_{i=1}^{N}\left\{x: x\cdot u_{i}\leq h(P,u_{i})+t\tau_{i}\right\}\in\mathcal{P}_{N}(u_{1},...,u_{N}).
$$
In particular, $h(P_t,u_{i})=h(P,u_{i})+t\tau_{i}$ for $i=1,\ldots,n$, and
Lemma 7.5.3 in Schneider \cite{SCH} yields that
$$
\frac{\partial V(P_{t})}{\partial t}=\sum_{i=1}^{N}\tau_{i}|F(P_t,u_i)|.
$$
Let $\lambda(t)=V(P_{t})^{-\frac{1}{n}}$. Then
$\lambda(t)P_{t}\in\mathcal{P}_{N}(u_{1},...,u_{N})$,
$V(\lambda(t)P_{t})=1$, $\lambda(t)$ is $C^1$ and 
\begin{equation}\tag{3.5}\label{Equation 3.5}
\lambda'(0)=-\frac{1}{n}\sum_{i=1}^{N}\tau_{i}|F(P,u_{i})|.
\end{equation}
Define $\xi(t):=\xi(\lambda(t)P_{t})$, and
\begin{equation}\tag{3.6}\label{Equation 3.6}
\begin{split}
\Phi(t)&:=\max_{\xi\in\lambda(t)P_{t}}\int_{S^{n-1}}\log\left(h(\lambda(t)P_{t},u)-\xi\cdot u\right)d\mu(u)\\
&=\sum_{k=1}^{N}\gamma_{k}\log(\lambda(t)h(P_{t},u_{k})-\xi(t)\cdot u_{k}).
\end{split}
\end{equation}
It follows from  Lemma~\ref{Lemma 3.3}, that
\begin{equation}\tag{3.7}\label{Equation 3.7}
\sum_{k=1}^{N}\gamma_{k}\frac{u_{k,i}}{\lambda(t)h(P_{t},u_{k})-\xi(t)\cdot
u_{k}}=0
\end{equation}
for $i=1,...,n$, where $u_{k}=(u_{k,1},...,u_{k,n})^{T}$.
In addition, since $\xi(P)$ is the origin, we have
\begin{equation}\tag{3.8}\label{Equation 3.8}
\sum_{k=1}^{N}\gamma_{k}\frac{u_{k}}{h(P,u_{k})}=0.
\end{equation}

Let $F=(F_1,\ldots,F_n)$ be a function from a small neighbourhood of the origin in $\mathbb{R}^{n+1}$ to $\mathbb{R}^n$ such that
$$
F_{i}(t,\xi_{1},...,\xi_{n})=\sum_{k=1}^{N}\gamma_{k}\frac{u_{k,i}}{\lambda(t)h(P_{t},u_{k})-(\xi_{1}u_{k,1}+...+\xi_{n}u_{k,n})}
$$
for $i=1,...,n$. Then,
\begin{eqnarray*}
\frac{\partial F_{i}}{\partial t}\bigg|_{(t, \xi_{1},...,\xi_{n})}&=&\sum_{k=1}^{N}\gamma_{k}\frac{-u_{k,i}(\lambda'(t)h(P_{t},u_{k})+\lambda(t)\tau_k)}
{[\lambda(t)h(P_{t},u_{k})-(\xi_{1}u_{k,1}+...+\xi_{n}u_{k,n})]^{2}}\\
\frac{\partial F_{i}}{\partial\xi_{j}}\bigg|_{(t, \xi_{1},...,\xi_{n})}&=&\sum_{k=1}^{N}\gamma_{k}\frac{u_{k,i}u_{k,j}}{[\lambda(t)h(P_{t},u_{k})-(\xi_{1}u_{k,1}+...+\xi_{n}u_{k,n})]^{2}}
\end{eqnarray*}
are continuous on a small neighborhood  of $(0,0,...,0)$ with
$$
\left(\frac{\partial F}{\partial\xi}\bigg|_{(0,...,0)}\right)_{n\times
n}=\sum_{k=1}^{N}\frac{\gamma_{k}}{h(P,u_{k})^{2}}u_{k}u_{k}^{T},
$$
where $u_{k}u_{k}^{T}$ is an $n\times n$ matrix. Since the unit vectors $u_{1},...,u_{N}$ are not concentrated on a closed hemisphere, $\mathbb{R}^{n}=$lin$\{u_{1},...,u_{N}\}$. Thus, for any $x\in\mathbb{R}^{n}$ with $x\neq0$, there exists a
$u_{i_{0}}\in\{u_{1},...,u_{N}\}$ such that
$u_{i_{0}}\cdot x\neq0$. Then,
\begin{equation*}
\begin{split}
x^{T}\left(\sum_{k=1}^{N}\frac{\gamma_{k}}{h(P,u_{k})^{2}}u_{k}u_{k}^{T}\right)x&=\sum_{k=1}^{N}\frac{\gamma_{k}}{h(P,u_{k})^{2}}(x\cdot u_{k})^{2}\\
&\geq\frac{\gamma_{i_{0}}}{h(P,u_{i_{0}})^{2}}(x\cdot u_{i_{0}})^{2}>0.
\end{split}
\end{equation*}
Therefore, $(\frac{\partial F}{\partial\xi}\big|_{(0,...,0)})$ is positive
definite. By this, the fact that $F_{i}(0,...,0)=0$ for $i=1,...,n$, the fact that $\frac{\partial F_{i}}{\partial\xi_{j}}$ is continuous on a neighborhood  of $(0,0,...,0)$ for all $1\leq i, j\leq n$ and the implicit function theorem, we have
$$
\xi'(0)=(\xi_{1}'(0),...,\xi_{n}'(0))
$$
exists.

From the fact that $\Phi(0)$ is a minimizer of $\Phi(t)$ (in Equation
(\ref{Equation 3.6})), Equation (\ref{Equation 3.5}), the fact that
$\sum_{k=1}^{N}\gamma_{k}=1$ and Equation (\ref{Equation 3.8}), we
have
\begin{equation*}
\begin{split}
0&=\Phi'(0)\\
&=\sum_{k=1}^{N}\gamma_{k}\frac{\lambda'(0)h(P,u_{k})+\lambda(0)\frac{dh(P_{t},u_{k})}{dt}\big|_{t=0}-\xi'(0)\cdot
u_{k}}{h(P,u_{k})}\\
&=\sum_{k=1}^{N}\gamma_{k}\frac{-\frac{1}{n}(\sum_{i=1}^{N}\tau_{i}|F(P,u_{i})|)h(P,u_{k})+\tau_{k}-\xi'(0)\cdot
u_{k}}{h(P,u_{k})}\\
&=-\sum_{i=1}^{N}\frac{|F(P,u_{i})|\tau_{i}}{n}+\sum_{k=1}^{N}\frac{\gamma_{k}\tau_{k}}{h(P,u_{k})}-\xi'(0)\cdot\left[\sum_{k=1}^{N}\gamma_{k}\frac{u_{k}}{h(P,u_{k})}\right]\\
&=\sum_{k=1}^{N}\left(\frac{\gamma_{k}}{h(P,u_{k})}-\frac{|F(P,u_{k})|}{n}\right)\tau_{k}.
\end{split}
\end{equation*}
Since $\tau_{1},...,\tau_{N}$ are arbitrary, we deduce that
$\gamma_{k}=\frac{1}{n}h(P,u_{k})|F(P,u_{k})|$ for $k=1,...,N$.
\end{proof}

\section{Existence of a solution of the extremal problem}

In this section, we prove Lemma~\ref{Lemma 4.4} about the existence of a solution of problem (\ref{Equation 3.2}) for the case where the discrete measure is not concentrated on any closed hemisphere of $S^{n-1}$ and satisfies the strict essential subspace concentration inequality.
Having the results of the previous section, the essential new ingredient is the following statement (see Lemma~\ref{Lemma 4.2}). 

If $\mu$ is a discrete measure on $S^{n-1}$ that is not concentrated on any closed hemisphere of $S^{n-1}$ and satisfies the strict essential subspace concentration inequality, and $\{P_m\}$ is a sequence of polytopes of unit volume such that
 the set of outer unit normals of $P_{m}$ is a subset of the support of $\mu$, and $\lim_{m\rightarrow\infty}d(P_{m})=\infty$
then
 $$
\lim_{m\to\infty}\Phi_{P_m}(\xi(P_m))=\infty.
 $$
It is equivalent to prove that any subsequence of $\{P_m\}$ has some subsequence $\{P_{m'}\}$ such that
$\lim_{m\to\infty}\Phi_{P_{m'}}(\xi(P_{m'}))=\infty$.

To indicate the idea, we sketch the argument for $n=2$.
 Let ${\rm supp}\,\mu=\{u_1,\ldots,u_N\}$, and let 
 $w_m=\min\{h_{P_m}(u)+h_{P_m}(-u):u\in S^1\}$ be the minimal width of $P_m$. Since $\lim_{m\rightarrow\infty}d(P_{m})=\infty$ and $V(P_m)=1$, we have $\lim_{m\rightarrow\infty}w_m=0$.
As $P_m$ is a polygon, we may assume that $w_m=h_{P_m}(u_1)+h_{P_m}(-u_1)$ possibly  after taking a subsequence and reindexing. If the angle of $u_1$ and $u_i$ is $\alpha_i\in(0,\pi)$ then $V_1(F(P_m,u_i))\leq w_m/\sin\alpha_i$, thus
$\lim_{m\rightarrow\infty}d(P_{m})=\infty$ implies that $-u_1\in{\rm supp}\,\mu$ for large $m$, say $u_2=-u_1$. Let $v\in S^1$ be orthogonal to $u_1$, and let $\gamma_i=\mu(\{u_i\})$ for $i=1,\ldots,N$. We may translate $P_m$ in a way such that $o\in\Int P_m$  in a way such that
 $h_{P_m}(u_1)=h_{P_m}(u_2)=w_m/2$, and
$h_{P_m}(v)= h_{P_m}(-v)$ hold for large $m$. Thus $V(P_m)=1$ yields the existence of a constant $c_1>0$ such that
  $h_{P_m}(u_i)>c_1/w_m$ for $i=3,\ldots,N$. Now ${\rm lin}\,u_1$ is an essential subspace with respect to $\mu$, and hence
$\gamma_1+\gamma_2<\gamma_3+\ldots+\gamma_N$
according to the strict essential subspace concentration inequality. Therefore
writing $c_2=\min\{2,c_1\}$, we have
\begin{eqnarray*}
\liminf_{m\to\infty}\exp\left(\Phi_{P_m}(\xi(P_m))\right)&\geq&
\liminf_{m\to\infty}\exp\left(\Phi_{P_m}(o)\right)=\liminf_{m\to\infty}\prod_{i=1}^N h_{P_m}(u_i)^{\gamma_i}\\
&\geq& \lim_{m\to\infty}
\left(\frac{w_m}2\right)^{\gamma_1+\gamma_2}\left(\frac{c_1}{w_m}\right)^{\gamma_3+\ldots+\gamma_N}
\geq \lim_{m\to\infty}\left(\frac{c_2}{w_m}\right)^{\gamma_3+\ldots+\gamma_N-\gamma_1-\gamma_2}=\infty.
\end{eqnarray*}

In the higher dimensional case, the idea is the very same. Only instead of one essential linear subspace like in the planar case,  we will find essential subspaces
 $X_0\subset\ldots\subset X_{q-1}$
in a way such that for $j=0,\ldots,q-1$, $P_m|_{X_j^\bot}$ is "much larger" than $P_m|_{X_j}$ for large $m$
 after taking suitable subsequence. 
 This is achieved in the preparatory statements Lemmas~\ref{prep41} to \ref{prep44}.

Given $N$ sequences, the first two observations will help to do book keeping of how the limits of the sequences compare. 

\begin{lemma}
\label{prep41}
Let $\{h_{1j}\}_{j=1}^{\infty},...,\{h_{Nj}\}_{j=1}^{\infty}$ be $N$ $(N\geq2)$ sequences of real numbers. Then, there exists a subsequence, $\{j_{n}\}_{n=1}^{\infty}$, of $\mathbb{N}$ and a rearrangement, $i_{1},...,i_{N}$, of $1,...,N$ such that
$$
h_{i_{1}j_{n}}\leq h_{i_{2}j_{n}}\leq...\leq h_{i_{N}j_{n}},
$$
for all $n\in\mathbb{N}$.
\end{lemma}
\begin{proof}
We prove it by induction on $N$. We first prove the case for $N=2$. For $j\in\mathbb{N}$, consider the sequence
$$
h_{j}=\max\{h_{1j}, h_{2j}\}.
$$
Since $\{h_{j}\}_{j=1}^{\infty}$ is an infinite sequence and $h_{j}$ either equals to $h_{1j}$ or equals to $h_{2j}$ for all $j\in\mathbb{N}$, there exists an $i_{2}\in\{1, 2\}$ and a subsequence, $\{j_{n}\}_{n=1}^{\infty}$, of $\mathbb{N}$ such that
$$
h_{j_{n}}=h_{i_{2}j_{n}}
$$
for all $n\in\mathbb{N}$. Let $i_{1}\in\{1, 2\}$ with $i_{1}\neq i_{2}$. Then,
$$
h_{i_{1}j_{n}}\leq h_{i_{2}j_{n}},
$$
for all $n\in\mathbb{N}$.

Suppose the lemma is true for $N=k$ (with $k\geq2$), we next prove that the lemma is true for $N=k+1$. For $j\in\mathbb{N}$, consider the sequence
$$
h_{j}=\max\{h_{1j}, h_{2j},...,h_{k+1 j}\}.
$$
Since $\{h_{j}\}_{j=1}^{\infty}$ is an infinite sequence and $h_{j}$ equals one of $h_{1j}, h_{2j},...,h_{k+1 j}$ for all $j\in\mathbb{N}$, there exists an $i_{k+1}\in\{1, 2,...,k+1\}$ and a subsequence, $\{j_{n}\}_{n=1}^{\infty}$, of $\mathbb{N}$ such that
$$
h_{j_{n}}=h_{i_{k+1}j_{n}}
$$
for all $n\in\mathbb{N}$.

Consider the sequences $\{h_{ij_{n}}\}_{n=1}^{\infty}$ ($1\leq i\leq k+1$ with $i\neq i_{k+1}$). By the inductive hypothesis, there exists a subsequence, $j_{n_{l}}$, of $j_{n}$ and a rearrangement, $i_{1},...,i_{k}$, of $1,...,\widehat{i_{k+1}},...,k+1$ such that
$$
h_{i_{1}j_{n_{l}}}\leq h_{i_{2}j_{n_{l}}}\leq...\leq h_{i_{k}j_{n_{l}}}
$$
for all $l\in\mathbb{N}$. By this and the fact that $h_{j_{n_{l}}}=h_{i_{k+1}j_{n_{l}}}$ for all $l\in\mathbb{N}$, we have
$$
h_{i_{1}j_{n_{l}}}\leq h_{i_{2}j_{n_{l}}}\leq...\leq h_{i_{k}j_{n_{l}}}\leq h_{i_{k+1}j_{n_{l}}}
$$
for all $l\in\mathbb{N}$.
\end{proof}

\begin{lemma}
\label{prep42}
Let $\{h_{1j}\}_{j=1}^{\infty},...,\{h_{Nj}\}_{j=1}^{\infty}$ be $N$ $(N\geq2)$ sequences of real numbers with
$$
h_{1j}\leq h_{2j}\leq...\leq h_{Nj}
$$
for all $j\in\mathbb{N}$, $\lim_{j\rightarrow\infty}h_{1j}=0$ and $\lim_{j\rightarrow\infty}h_{Nj}=\infty$.
Then, there exist $q\geq1$,
$$
1=\alpha_{0}<\alpha_{1}<...<\alpha_{q}\leq N<N+1=\alpha_{q+1}
$$
and a subsequence, $\{j_{n}\}_{n=1}^{\infty}$, of $\mathbb{N}$ such that if $i=1,...,q$, then
$$
\lim_{n\rightarrow\infty}\frac{h_{\alpha_{i}j_{n}}}{h_{\alpha_{i-1}j_{n}}}=\infty,
$$
if $i=0,...,q,$ and $\alpha_{i}\leq k\leq\alpha_{i+1}-1$, then
$$
\lim_{n\rightarrow\infty}\frac{h_{kj_{n}}}{h_{\alpha_{i}j_{n}}}
$$
exists and equals to a positive number.
\end{lemma}

\begin{proof}
Let $\alpha_{0}=1$. By conditions,
$$
\frac{h_{1j}}{h_{1j}}\leq\frac{h_{2j}}{h_{1j}}\leq...\leq\frac{h_{Nj}}{h_{1j}},
$$
$\overline{\lim}_{j\rightarrow\infty}\frac{h_{ij}}{h_{1j}}$ either exists (equals to a positive number) or goes to $\infty$, and $\overline{\lim}_{j\rightarrow\infty}\frac{h_{Nj}}{h_{1j}}=\infty.$ Thus, there exists an $\alpha_{1}$ ($1<\alpha_{1}\leq N$) such that for $1\leq i\leq\alpha_{1}-1$,
$$
\overline{\lim}_{j\rightarrow\infty}\frac{h_{ij}}{h_{1j}}<\infty
$$
and
$$
\overline{\lim}_{j\rightarrow\infty}\frac{h_{\alpha_{1}j}}{h_{1j}}=\infty.
$$
Hence, we can choose a subsequence, $\{j_{n}'\}_{n=1}^{\infty}$, of $\mathbb{N}$ such that
$$
\lim_{n\rightarrow\infty}\frac{h_{\alpha_{1}j_{n}'}}{h_{1j_{n}'}}=\infty,
$$
and for $1\leq i\leq\alpha_{1}-1$,
$$
\overline{\lim}_{n\rightarrow\infty}\frac{h_{ij_{n}'}}{h_{1j_{n}'}}\leq\overline{\lim}_{j\rightarrow\infty}\frac{h_{ij}}{h_{1j}}<\infty.
$$

By choosing $\alpha_{1}-2$ times subsequences of $j_{n}'$, we can find a subsequence, $\{j_{n}''\}_{n=1}^{\infty}$, of $\{j_{n}'\}_{n=1}^{\infty}$ such that
$$
\lim_{n\rightarrow\infty}\frac{h_{\alpha_{1}j_{n}''}}{h_{1j_{n}''}}=\infty,
$$
and for $1\leq i\leq\alpha_{1}-1$,
$$
\lim_{n\rightarrow\infty}\frac{h_{ij_{n}''}}{h_{1j_{n}''}}
$$
exists and equals to a positive number.

By repeating (at most $N-\alpha_{1}$ times) similar arguments for the sequences $\{h_{ij''_{n}}\}_{n=1}^{\infty}$ $(\alpha_{1}\leq i\leq N)$, we can find $q\geq1$,
$$
1=\alpha_{0}<\alpha_{1}<...<\alpha_{q}\leq N<N+1=\alpha_{q+1}
$$
and a subset, $\{j_{n}\}_{n=1}^{\infty}$, of $\mathbb{N}$ that satisfy the conditions in the lemma.
\end{proof}

The following lemma compares positive hull and linear hull.

\begin{lemma}
\label{prep43}
Suppose $u_{1},...,u_{l}\in S^{d-1}$ ($d\geq2$), $\mathbb{R}^{d}=\textmd{lin}\{u_{1},...,u_{l}\}$, and $u_{1},...,u_{l}$ are not concentrated on a closed hemisphere of $S^{d-1}$, then
$$
\mathbb{R}^{d}=\textmd{pos}\{u_{1},...,u_{l}\}.
$$
Moreover, there exists $\lambda>0$ depending on $u_{1},...,u_{l}$ such that
 any $u\in S^{d-1}$ can be written in the form
$$
u=a_{i_{1}}u_{i_{1}}+...+a_{i_{d}}u_{i_{d}}
$$
where $\{u_{i_{1}},...,u_{i_{d}}\}\subset \{u_{1},...,u_{l}\}$ and $0\leq a_{i_{1}},...,a_{i_{d}}\leq\lambda$.
\end{lemma}
\begin{proof} Let $Q$ be the convex hull of $\{u_{1},...,u_{l}\}$, which is a polytope. Since $u_{1},...,u_{l}$ are not concentrated on a closed hemisphere of $S^{d-1}$, the origin is an interior point of $Q$. In particular, $rB^d\subset Q$ for some $r>0$.

For $u\in S^{d-1}$, there exists some $t\geq r$ such that $tu\in \partial Q$. It follows that $tu\in F$ for some facet $F$ of $Q$. We deduce from the Charateodory theorem that there exists vertices $u_{i_{1}},...,u_{i_{d}}$ of $F$ that $tu$ lies in their convex hull. In other words,
$$
tu=\alpha_{i_{1}}u_{i_{1}}+...+\alpha_{i_{d}}u_{i_{d}}
$$
where $\alpha_{i_{1}},\ldots,\alpha_{i_{d}}\geq 0$ and $\alpha_{i_{1}}+\ldots+\alpha_{i_{d}}=1$. Therefore we choose
$a_{i_{j}}=\alpha_{i_{j}}/t\leq 1/r$ for $j=1,\ldots,d$, which in turn satisfy
$u=a_{i_{1}}u_{i_{1}}+...+a_{i_{d}}u_{i_{d}}$.
In particular, we may take $\lambda=1/r$.
\end{proof}

The following lemma will be the last preparatory statement.

\begin{lemma}
\label{prep44}
Suppose $\mu$ is a discrete measure on $S^{n-1}$ that is not concentrated on any closed hemisphere of $S^{n-1}$ with supp($\mu$)$=\{u_{1},...,u_{N}\}$ and $\mu(u_{i})=\gamma_{i}$ for $i=1,...,N$. If $P_{m}$ is a sequence of polytopes with $V(P_{m})=1$, $\xi(P_{m})$ is the origin, the set of outer unit normals of $P_{m}$ is a subset of $\{u_{1},...,u_{N}\}$, $\lim_{m\rightarrow\infty}d(P_{m})=\infty$ and
$$
h(P_{m},u_{1})\leq h(P_{m},u_{2})\leq...\leq h(P_{m},u_{N})
$$
for all $m\in\mathbb{N}$. Then, there exist $q\geq1$, and $1=\alpha_{0}<\alpha_{1}<...<\alpha_{q}\leq N<N+1=\alpha_{q+1}$ such that if $j=1,...,q$, then
\begin{equation}\tag{4.0a}\label{Equation 4.0a}
\lim_{m\rightarrow\infty}\frac{h(P_{m},u_{\alpha_{j}})}{h(P_{m},u_{\alpha_{j-1}})}=\infty,
\end{equation}
and if $j=0,...,q$ and $\alpha_{j}\leq k\leq\alpha_{j+1}-1$, then
\begin{equation}\tag{4.0b}\label{Equation 4.0b}
\lim_{m\rightarrow\infty}\frac{h(P_{m},u_{k})}{h(P_{m},u_{\alpha_{j}})}=t_{kj}<\infty.
\end{equation}
Moreover, $X_{j}=\textmd{pos}\{u_{1},...,u_{\alpha_{j+1}-1}\}$ are subspaces of $\mathbb{R}^{n}$ for all $0\leq j\leq q$ and
$$
1\leq\dim(X_{0})<\dim(X_{1})<...<\dim(X_{q})=n.
$$
\end{lemma}
\begin{proof}
By the conditions that $\lim_{m\rightarrow\infty}d(P_{m})=\infty$, $V(K)=1$ and $h(P_{m},u_{1})\leq h(P_{m},u_{2})\leq...\leq h(P_{m},u_{N})$ for all $m\in\mathbb{N}$, we have,
$$
\lim_{m\rightarrow\infty}h(P_{m},u_{1})=0\textmd{ and }\lim_{m\rightarrow\infty}h(P_{m},u_{N})=\infty.
$$
From Lemma~\ref{prep42}, we may assume that there exist $q\geq1$, and
$$
1=\alpha_{0}<\alpha_{1}<...<\alpha_{q}\leq N<N+1=\alpha_{q+1}
$$
that satisfy Equations (\ref{Equation 4.0a}) and (\ref{Equation 4.0b}).

For $j=0,...,q-1$, we consider the cone
$$
\Sigma_{j}=\textmd{pos}\{u_{1},...,u_{\alpha_{j+1}-1}\},
$$
and its negative polar
$$
\Sigma_{j}^{*}=\{v\in\mathbb{R}^{n}: v\cdot u_{i}\leq0\textmd{ for all }i=1,...,\alpha_{j+1}-1\}.
$$

Let $0\leq j\leq q-1$, $1\leq p\leq\alpha_{j+1}-1$ and $v\in\Sigma_{j}^{*}\cap S^{n-1}$. From the condition that $\xi(P_{m})$ is the origin and Lemma~\ref{Lemma 3.3},
$$
\sum_{i=1}^{N}\frac{\gamma_{i}(v\cdot u_{i})}{h(P_{m},u_{i})}=0.
$$
By this and the fact that $v\in\Sigma_{j}^{*}\cap S^{n-1}$,
\begin{equation*}
\begin{split}
0&\geq\gamma_{p}(v\cdot u_{p})=-\sum_{i\neq p}\frac{h(P_{m},u_{p})}{h(P_{m},u_{i})}\gamma_{i}(v\cdot u_{i})\\
&\geq-\sum_{i\geq\alpha_{j+1}}\frac{h(P_{m},u_{p})}{h(P_{m},u_{i})}\gamma_{i}(v\cdot u_{i})\\
&\geq-\sum_{i\geq\alpha_{j+1}}\frac{h(P_{m},u_{p})}{h(P_{m},u_{i})}\gamma_{i}.
\end{split}
\end{equation*}
By this, (4.0a) and (4.0b), we have, $\gamma_{p}(v\cdot u_{p})$ is no bigger than 0, and no less than any negative number. Thus,
$$
v\cdot u_{p}=0
$$
for all $p=1,...,\alpha_{j+1}-1$ and $v\in\Sigma_{j}^{*}\cap S^{n-1}$. Then, for any $u\in\textmd{lin}\{u_{1},...,u_{\alpha_{j+1}-1}\}$ and $v\in\Sigma_{j}^{*}$, $u\cdot v=0$. Hence,
$$
\Sigma_{j}^{*}\cap\textmd{lin}\{u_{1},...,u_{\alpha_{j+1}-1}\}=\{0\}.
$$

We claim that $\{u_{1},...,u_{\alpha_{j+1}-1}\}$ is not concentrated on a closed hemisphere of $S^{n-1}\cap\textmd{lin}\{u_{1},...,u_{\alpha_{j+1}-1}\}$. Otherwise, there exists a vector $u_{0}\in\textmd{lin}\{u_{1},...,u_{\alpha_{j+1}-1}\}$ such that $u_{0}\neq0$ and $u_{0}\cdot u_{p}\leq0$ for all $p=1,...,\alpha_{j+1}-1$. This contradicts the fact that $\Sigma_{j}^{*}\cap\textmd{lin}\{u_{1},...,u_{\alpha_{j+1}-1}\}=\{0\}.$ Hence, $\{u_{1},...,u_{\alpha_{j+1}-1}\}$ is not concentrated on a closed hemisphere of $S^{n-1}\cap\textmd{lin}\{u_{1},...,u_{\alpha_{j+1}-1}\}$. By Lemma~\ref{prep43},
$$
\textmd{lin}\{u_{1},...,u_{\alpha_{j+1}-1}\}=\textmd{pos}\{u_{1},...,u_{\alpha_{j+1}-1}\}.
$$

Let $X_{j}=\textmd{pos}\{u_{1},...,u_{\alpha_{j+1}-1}\}$, $d_{j}=\dim X_{j}$ for $j=0,...,q$, and $d_{-1}=0$. Obviously, $d_{0}\geq1$ and $d_{q}=n$. We claim that $d_{0}<d_{1}<...<d_{q}$. Otherwise, there exist $0\leq k<l\leq q$ such that $d_{k}=d_{l}$, and thus $X_{k}=X_{l}$. We write $\lambda>0$ for the constant of Lemma~\ref{prep43} depending on 
$u_1,\ldots,u_N$. 
By Lemma~\ref{prep43}, there exist $u_{i_{1}},...,u_{i_{d_{k}}}\in\{u_{1},...,u_{\alpha_{k+1}-1}\}$ and $0\leq a_{i_{1}},...,a_{i_{d_{k}}}\leq\lambda$ such that
$$
u_{\alpha_{l}}=a_{i_{1}}u_{i_{1}}+...+a_{i_{d_{k}}}u_{i_{d_{k}}}.
$$
Hence,
\begin{equation*}
\begin{split}
h(P_{m},u_{\alpha_{l}})&=h(P_{m},a_{i_{1}}u_{i_{1}}+...+a_{i_{d_{k}}}u_{i_{d_{k}}})\\
&\leq a_{i_{1}}h(P_{m},u_{i_{1}})+...+a_{i_{d_{k}}}h(P_{m},u_{i_{d_{k}}}),
\end{split}
\end{equation*}
for all $m\in\mathbb{N}$. But this contradicts (4.0a) and (4.0b). Therefore,
$$
1\leq d_{0}<d_{1}<...<d_{q}=n.
$$
\end{proof}

\begin{lemma}\label{Lemma 4.2}
 Suppose $\mu$ is a discrete measure on $S^{n-1}$ that is not concentrated on any closed hemisphere of $S^{n-1}$, and satisfies the strict essential subspace concentration inequality. If $P_{m}$ is a sequence of polytopes with $V(P_{m})=1$, $\xi(P_{m})$ is the origin, the set of outer unit normals of $P_{m}$ is a subset of the support of $\mu$ and $\lim_{m\rightarrow\infty}d(P_{m})=\infty$, then
 $$
 \int_{S^{n-1}}\log h(P_{m},u)d\mu(u)
 $$
 is not bounded from above.
\end{lemma}
\begin{proof}
Without loss of generality, we can suppose $|\mu|=1$. Let supp$(\mu)=\{u_{1},...,u_{N}\}$, and $\mu(\{u_{i}\})=\gamma_{i}, i=1,...,N$. From Lemma~\ref{prep41}, we may assume that
\begin{equation}\tag{4.1}\label{Equation 4.1}
h(P_{m},u_{1})\leq...\leq h(P_{m},u_{N}),
\end{equation}
for all $m\in\mathbb{N}$.
Since $\lim_{m\rightarrow\infty}d(P_{m})=\infty$ and $V(K)=1$,
$$
\lim_{m\rightarrow\infty}h(P_{m},u_{1})=0\textmd{ and }\lim_{m\rightarrow\infty}h(P_{m},u_{N})=\infty.
$$
By Lemma~\ref{prep44}, there exist $q\geq1$, and
$$
1=\alpha_{0}<\alpha_{1}<...<\alpha_{q}\leq N<N+1=\alpha_{q+1}
$$
such that if $j=1,...,q$, then
\begin{equation}\tag{4.2a}\label{Equation 4.2}
\lim_{m\rightarrow\infty}\frac{h(P_{m},u_{\alpha_{j}})}{h(P_{m},u_{\alpha_{j-1}})}=\infty,
\end{equation}
and if $j=0,...,q$ and $\alpha_{j}\leq k\leq\alpha_{j+1}-1$, then
\begin{equation}\tag{4.2b}\label{Equation 4.3}
\lim_{m\rightarrow\infty}\frac{h(P_{m},u_{k})}{h(P_{m},u_{\alpha_{j}})}=t_{k,j}<\infty.
\end{equation}
Moreover, $X_{j}=\textmd{pos}\{u_{1},...,u_{\alpha_{j+1}-1}\}$ are subspaces of $\mathbb{R}^{n}$ with respect to $\mu$ for all $0\leq j\leq q$ with
$$
1\leq d_{0}<d_{1}<...<d_{q}=n,
$$
where $d_{j}=\dim(X_{j})$. In particular, $X_0,\ldots,X_{q-1}$ are essential subspaces.

Let $\widetilde{X}_{0}=X_{0}$, and if $j=1,...,q$, then let
$$
\widetilde{X}_{j}=X_{j-1}^{\bot}\cap X_{j}.
$$
From the definition of $X_{j}$ and $\widetilde{X}_{j}$, we have, $\tilde{X}_{j_{1}}\perp\tilde{X}_{j_{2}}$ for $j_{1}\neq j_{2}$, $\dim\widetilde{X}_{j}=d_{j}-d_{j-1}>0$ for $j=0,...,q$, and $\mathbb{R}^{n}$ is a direct sum of $\tilde{X}_{0},...,\tilde{X}_{q}$.

Let $\lambda>0$ be the constant of  Lemma~\ref{prep43} for $u_1,\ldots,u_N$.
Suppose $0\leq j\leq q$ and $u\in X_{j}\cap S^{n-1}$. By Lemma~\ref{prep43}, there exists a subset, $\{u_{i_{1}},...,u_{i_{d_{j}}}\}$, of $\{u_{1},...,u_{\alpha_{j+1}-1}\}$ and $0\leq a_{i_{1}},...,a_{i_{d_{j}}}\leq\lambda$ such that
$$
u=a_{i_{1}}u_{i_{1}}+...+a_{i_{d_{j}}}u_{i_{d_{j}}}.
$$
Then,
\begin{equation*}
\begin{split}
h(P_{m},u)&=h(P_{m},a_{i_{1}}u_{i_{1}}+...+a_{i_{d_{j}}}u_{i_{d_{j}}})\\
&\leq a_{i_{1}}h(P_{m},u_{i_{1}})+...+a_{i_{d_{j}}}h(P_{m},u_{i_{d_{j}}}).
\end{split}
\end{equation*}
By this, (4.2a) and (4.2b), if $m$ is large, then
$$
h(P_{m},u)\leq t_{j}h(P_{m},u_{\alpha_{j}})\textmd{ for all }u\in X_{j}\cap S^{n-1}
$$
where $t_{j}=d_j\lambda (t_{\alpha_{j+1}-1,j}+1)>0$.
Hence, for $j=0,...,q$,
$$
P_{m}|_{\tilde{X}_{j}}\subset t_{j}h(P_{m},u_{\alpha_{j}})\big(B^{n}\cap\tilde{X}_{j}\big).
$$
By this and the fact that $\mathbb{R}^{n}$ is a direct sum of $\tilde{X}_{0},...,\tilde{X}_{q}$,
$$
P_{m}\subset\sum_{j=0}^{q}t_{j}h(P_{m},u_{\alpha_{j}})\big(B^{n}\cap\tilde{X}_{j}\big),
$$
where the summation is Minkowski sum.
Let
$$
\omega=\max_{0\leq j\leq q}t_{j}\kappa_{d_{j}-d_{j-1}}^{\frac{1}{d_{j}-d_{j-1}}},
$$
where $\kappa_{d_{j}-d_{j-1}}$ is the volume of the $(d_{j}-d_{j-1})$-dimensional unit ball. Then, for $j=0,...,q$
$$
V_{d_{j}-d_{j-1}}\left(t_{j}h(P_{m},u_{\alpha_{j}})\big(B^{n}\cap\tilde{X}_{j}\big)\right)\leq(\omega h(P_{m},u_{\alpha_{j}}))^{d_{j}-d_{j-1}}.
$$
From this, the fact that $\mathbb{R}^{n}$ is a direct sum of $\tilde{X}_{0},...,\tilde{X}_{q}$, and Fubini's formula, we have
\begin{equation*}
\begin{split}
1&=V(P_{m})\\
&\leq V\bigg(\sum_{j=0}^{q}t_{j}h(P_{m},u_{\alpha_{j}})\big(B^{n}\cap\tilde{X}_{j}\big)\bigg)\\
&=\prod_{j=0}^{q}V_{d_{j}-d_{j-1}}\left(t_{j}h(P_{m},u_{\alpha_{j}})\big(B^{n}\cap\tilde{X}_{j}\big)\right)\\
&\leq\prod_{j=0}^{q}(\omega h(P_{m},u_{\alpha_{j}}))^{d_{j}-d_{j-1}}.
\end{split}
\end{equation*}
It follows from $0=d_{-1}<d_{0}<...<d_{q}=n$ that if $m$ is large, then
$$
\sum_{j=0}^{q}\left(\frac{d_{j}}{n}-\frac{d_{j-1}}{n}\right)\log h(P_{m},u_{\alpha_{j}})\geq-\log\omega.
$$
We rewrite the last inequality as
\begin{equation}\tag{4.3}
\log h(P_{m},u_{\alpha_{q}})\geq-\sum_{j=0}^{q-1}\frac{d_{j}}{n}\log\frac{h(P_{m},u_{\alpha_{j}})}{h(P_{m},u_{\alpha_{j+1}})}-\log\omega.
\end{equation}

For $j=0,...,q$, we set $\beta_{j}=\mu(X_{j}\cap S^{n-1})=\sum_{i=1}^{\alpha_{j+1}-1}\gamma_{i}$, and $\beta_{-1}=0$. We deduce from the facts that $X_{j}$ is an essential subspace with $d_{j}=\dim(X_{j})$, and from the condition that $\mu$ satisfies the strict essential subspace concentration condition that
\begin{equation}\tag{4.4}\label{Equation 4.4}
\beta_{j}<\frac{d_{j}}{n} \mbox{ \ \  for $0\leq j\leq q-1.$}
\end{equation}

By the fact that $h(P_{m},u_{1})\leq h(P_{m},u_{2})\leq...\leq h(P_{m},u_{N})$, the fact that $\beta_{q}=1$ and (4.3),
\begin{equation*}
\begin{split}
\sum_{i=1}^{N}\gamma_{i}\log h(P_{m},u_{i})&=\sum_{i=1}^{\alpha_{1}-1}\gamma_{i}\log h(P_{m},u_{i})+\sum_{i=\alpha_{1}}^{\alpha_{2}-1}\gamma_{i}\log h(P_{m},u_{i})+...+\sum_{i=\alpha_{q}}^{N}\gamma_{i}\log h(P_{m},u_{i})\\
&\geq\sum_{i=1}^{\alpha_{1}-1}\gamma_{i}\log h(P_{m},u_{\alpha_{0}})+\sum_{i=\alpha_{1}}^{\alpha_{2}-1}\gamma_{i}\log h(P_{m},u_{\alpha_{1}})+...+\sum_{i=\alpha_{q}}^{N}\gamma_{i}\log h(P_{m},u_{\alpha_{q}})\\
&=\sum_{j=0}^{q}(\beta_{j}-\beta_{j-1})\log h(P_{m},u_{\alpha_{j}})\\
&=\log h(P_{m},u_{\alpha_{q}})+\sum_{j=0}^{q-1}\beta_{j}\log\frac{h(P_{m},u_{\alpha_{j}})}{h(P_{m},u_{\alpha_{j+1}})}\\
&\geq-\log\omega+\sum_{j=0}^{q-1}\left(\beta_{j}-\frac{d_{j}}{n}\right)\log\frac{h(P_{m},u_{\alpha_{j}})}{h(P_{m},u_{\alpha_{j+1}})}.
\end{split}
\end{equation*}
It follows from (\ref{Equation 4.1}), (4.2a), (\ref{Equation 4.4}) that for $j=0,...,q-1$,
$$
\lim_{m\rightarrow\infty}\left(\beta_{j}-\frac{d_{j}}{n}\right)\log\frac{h(P_{m},u_{\alpha_{j}})}{h(P_{m},u_{\alpha_{j+1}})}=\infty.
$$
Therefore,
$$
\lim_{m\rightarrow\infty}\sum_{i=1}^{N}\gamma_{i}\log h(P_{m},u_{i})=\infty.
$$
\end{proof}

The following lemma will be needed (see, \cite{Z4}, Lemma 3.5).

\begin{lemma}\label{Lemma 4.3}
If $P$ is a polytope in $\mathbb{R}^{n}$ and $v_{0}\in S^{n-1}$ with $V_{n-1}(F(P,v_{0}))=0$, then there exists a $\delta_{0}>0$ such that for $0\leq\delta<\delta_{0}$
$$
V(P\cap\{x: x\cdot v_{0}\geq h(P,v_{0})-\delta\})=c_{n}\delta^{n}+...+c_{2}\delta^{2},
$$
where $c_{n},...,c_{2}$ are constants that depend on $P$ and $v_{0}$.
\end{lemma}

Now, we have prepared enough to prove the main result of this section.

\begin{lemma}\label{Lemma 4.4}
Suppose the discrete measure $\mu=\sum_{k=1}^{N}\gamma_{k}\delta_{u_{i}}$ is not concentrated on a closed hemisphere. If $\mu$ satisfies the strict essential subspace concentration inequality, then there exists a
$P\in\mathcal{P}_{N}(u_{1},...,u_{N})$ such that $\xi(P)=0$,
$V(P)=|\mu|$ and
$$
\Phi_{P}(0)=\inf\left\{\max_{\xi\in\Int (Q)}\Phi_{Q}(\xi):
Q\in\mathcal{P}(u_{1},...,u_{N})\textmd{ and
}V(Q)=|\mu|\right\},
$$
where $\Phi_{Q}(\xi)=\int_{S^{n-1}}\log\left(h(Q,u)-\xi\cdot u\right)d\mu(u)$.
\end{lemma}
\begin{proof}
It is easily seen that it is sufficient to establish the lemma under
the assumption that $|\mu|=1$.

Obviously, for $P, Q\in\mathcal{P}(u_{1},...,u_{N})$, if there exists an $x\in\mathbb{R}^{n}$ such that $P=Q+x$, then
$$
\Phi_{P}(\xi(P))=\Phi_{Q}(\xi(Q)).
$$
Thus, we can choose a sequence $P_{i}\in\mathcal{P}(u_{1},...,u_{N})$ with $\xi(P_{i})=0$ and $V(P_{i})=1$ such that $\Phi_{P_{i}}(0)$ converges
to
$$
\inf\left\{\max_{\xi\in\Int (Q)}\Phi_{Q}(\xi):
Q\in\mathcal{P}(u_{1},...,u_{N})\textmd{ and }V(Q)=1\right\}.
$$

Choose a fixed $P_{0}\in\mathcal{P}(u_{1},...,u_{N})$ with $V(P_{0})=1$, then
$$
\inf\left\{\max_{\xi\in\Int (Q)}\Phi_{Q}(\xi):
Q\in\mathcal{P}(u_{1},...,u_{N})\textmd{ and }V(Q)=1\right\}\leq\Phi_{P_{0}}(\xi(P_{0})).
$$

We claim that $P_{i}$ is bounded. Otherwise, from Lemma~\ref{Lemma 4.2}, $\Phi_{P_{i}}(\xi(P_{i}))$ is not bounded from above. This contradicts the previous inequality. Therefore, $P_{i}$ is bounded.

From Lemma~\ref{Lemma 3.2} and the Blaschke selection theorem, there exists a subsequence of $P_{i}$ that converges to a
polytope $P$ such that $P\in\mathcal{P}(u_{1},...,u_{N})$, $V(P)=1$, $\xi(P)=0$ and
\begin{equation}\tag{4.5}\label{Equation 4.5}
\Phi_{P}(0)=\inf\left\{\max_{\xi\in\Int (Q)}\Phi_{Q}(\xi):
Q\in\mathcal{P}(u_{1},...,u_{N})\textmd{ and }V(Q)=1\right\}.
\end{equation}

We next prove that $F(P,u_{i})$ are facets for all $i=1,...,N$. Otherwise, there exists an $i_{0}\in\{1,...,N\}$ such that
$$
F(P,u_{i_{0}})
$$
is not a facet of $P$.

Choose $\delta>0$ small enough so that the polytope
$$
P_{\delta}=P\cap\{x: x\cdot u_{i_{0}}\leq h(P,u_{i_{0}})-\delta\}\in\mathcal{P}(u_{1},...,u_{N}),
$$
and (by Lemma~\ref{Lemma 4.3})
$$
V(P_{\delta})=1-(c_{n}\delta^{n}+...+c_{2}\delta^{2}),
$$
where $c_{n},...,c_{2}$ are constants that depend on $P$ and direction $u_{i_{0}}$.

From Lemma~\ref{Lemma 3.2}, for any $\delta_{i}\rightarrow0$ $\xi(P_{\delta_{i}})\rightarrow 0$. We have,
$$
\lim_{\delta\rightarrow0}\xi(P_{\delta})=0.
$$
Let $\delta$ be small enough so that $h(P,u_{k})>\xi(P_{\delta})\cdot u_{k}+\delta$ for all $k\in\{1,...,N\}$, and let
$$
\lambda=V(P_{\delta})^{-\frac{1}{n}}=(1-(c_{n}\delta^{n}+...+c_{2}\delta^{2}))^{-\frac{1}{n}}.
$$
From this and Equation (\ref{Equation 3.3}), we have
\begin{equation*}
\begin{split}
\prod_{k=1}^{N}\left(h(\lambda P_{\delta},u_{k})-\xi(\lambda
P_{\delta})\cdot
u_{k}\right)^{\gamma_{k}}&=\lambda\prod_{k=1}^{N}\left(h(P_{\delta},u_{k})-\xi(P_{\delta})\cdot
u_{k}\right)^{\gamma_{k}}\\
&=\lambda\left[\prod_{k=1}^{N}\left(h(P,u_{k})-\xi(P_{\delta})\cdot
u_{k}\right)^{\gamma_{k}}\right]\left[\frac{h(P,u_{i_{0}})-\xi(P_{\delta})\cdot
u_{i_{0}}-\delta}{h(P,u_{i_{0}})-\xi(P_{\delta})\cdot
u_{i_{0}}}\right]^{\gamma_{i_{0}}}\\
&=\left[\prod_{k=1}^{N}\left(h(P,u_{k})-\xi(P_{\delta})\cdot
u_{k}\right)^{\gamma_{k}}\right]\frac{(1-\frac{\delta}{h(P,u_{i_{0}})-\xi(P_{\delta})\cdot
u_{i_{0}}})^{\gamma_{i_{0}}}}{(1-(c_{n}\delta^{n}+...+c_{2}\delta^{2}))^{\frac{1}{n}}}\\
&\leq\left[\prod_{k=1}^{N}\left(h(P,u_{k})-\xi(P_{\delta})\cdot
u_{k}\right)^{\gamma_{k}}\right]\frac{(1-\frac{\delta}{d_{0}})^{\gamma_{i_{0}}}}{(1-(c_{n}\delta^{n}+...+c_{2}\delta^{2}))^{\frac{1}{n}}},
\end{split}
\end{equation*}
where $d_{0}=d(P)$ is the diameter of $P$.
Thus,
\begin{equation}\tag{4.6}\label{Equation 4.6}
\Phi_{\lambda P_{\delta}}\left(\xi(\lambda P_{\delta})\right)\leq\Phi_{P}\left(\xi(P_{\delta})\right)+B(\delta),
\end{equation}
where
\begin{equation}\tag{4.7}\label{Equation 4.7}
B(\delta)=\gamma_{i_{0}}\log\left(1-\frac{\delta}{d_{0}}\right)-\frac{1}{n}\log\left(1-(c_{n}\delta^{n}+...+c_{2}\delta^{2})\right).
\end{equation}
Obviously,
\begin{equation}\tag{4.8}\label{Equation 4.8}
B'(\delta)=\gamma_{i_{0}}\frac{-1/d_{0}}{1-\delta/d_{0}}+\frac{1}{n}\frac{nc_{n}\delta^{n-1}+...+2c_{2}\delta}{1-(c_{n}\delta^{n}+...+c_{2}\delta^{2})}<0,
\end{equation}
when the positive $\delta$ is small enough. From this and the fact that $B_{1}(0)=0$,
$$
B(\delta)<0
$$
when the positive $\delta$ is small enough.

From this and Equations (\ref{Equation 4.6}), (\ref{Equation 4.7}), (\ref{Equation 4.8}), there
exists a $\delta_{0}>0$ such that $P_{\delta_{0}}\in\mathcal{P}(u_{1},...,u_{N})$ and
$$
\Phi_{\lambda_{0}P_{\delta_{0}}}(\xi(\lambda_{0}P_{\delta_{0}}))<\Phi_{P}(\xi(P_{\delta_{0}}))\leq\Phi_{P}(\xi(P))=\Phi_{P}(0),
$$
where $\lambda_{0}=V(P_{\delta_{0}})^{-\frac{1}{n}}$. Let $P_{0}=\lambda_{0}P_{\delta_{0}}-\xi(\lambda_{0}P_{\delta_{0}})$, then
$P_{0}\in\mathcal{P}(u_{1},...,u_{N})$, $V(P_{0})=1$, $\xi(P_{0})=0$ and
\begin{equation*}
\Phi_{P_{0}}(0)<\Phi_{P}(0).
\end{equation*}
This contradicts Equation (\ref{Equation 4.5}). Therefore, $P\in\mathcal{P}_{N}(u_{1},...,u_{N})$.
\end{proof}

\section{Existence of the solution to the discrete logarithmic Minkowski problem}

If $\mu$ is a Borel measure on $S^{n-1}$ and $\xi$ is a proper subspace of $\mathbb{R}^{n}$, it will be convenient to write $\mu_{\xi}$ for the restriction of $\mu$ to $S^{n-1}\cap\xi$.
In this section, we prove the main result Theorem~\ref{sccgen} of this paper based on the folowing idea. Let $\mu$ be  discrete measure on $S^{n-1}$, $n\geq 2$, that is not concentrated on any closed hemisphere and satisfies the essential subspace concentration condition.
If $\mu$ satisfies the strict essential subspace concentration inequality, 
then Lemma~\ref{Lemma 4.4} yields that $\mu$ is a cone volume measure. 
Otherwise there exist complementary proper subspaces $\xi$ and $\xi'$ such that
${\rm supp}\,\mu=S^{n-1}\cap(\xi\cup \xi')$, and $\mu_\xi$ and 
$\mu_{\xi}'$ are not concentrated on any closed hemisphere of $\xi\cap S^{n-1}$ and  $\xi'\cap S^{n-1}$, respectively, and satisfy the essential subspace concentration condition. 
Therefore $\mu_\xi$ and 
$\mu_{\xi}'$ are cone volume measures on $\xi\cap S^{n-1}$ and $\xi'\cap S^{n-1}$,  respectively, by induction on the dimension of the ambient space, which in turn imply that $\mu$ is a cone volume measure.

However, it is possible that $\dim \xi=1$. Therefore in order to execute the plan, 
  we extend the notions occuring in Theorem~\ref{sccgen} to $\mathbb{R}^1$.
 The role of a compact convex set containing the origin in its interior is played by some interval $K=[a,b]$ with $a<0$ and $b>0$, and closed hemispheres of $S^0=\{-1,1\}$ are $\{1\}$ and 
$\{-1\}$. The cone volume measure on $S^0$ associated to $K$  satisfies
$V_K(\{-1\})=|a|$ and $V_K(\{1\})=b$. In addition, we say that a non-trivial measure $\mu$ on $S^0$ satisfies 
the essential subspace concentration inequality if it is not concentrated on any closed hemisphere; namely, if $\mu(\{-1\})>0$ and $\mu(\{1\})>0$. These notions are in accordance with Definition~\ref{essentialdef} because if $n=1$, then there 
is no subspace $\xi$ such that $0<\dim\xi<n$.

We note that the notion of strict essential subspace concentration inequality is defined and used
 only if the dimension $n\geq 2$.

The following lemma will be needed. The proof is the same that of Lemma 7.1 in \cite{BLYZ}.

\begin{lemma}\label{Lemma 5.1}
Suppose $n\geq2$, $\mu$ is a discrete measure on $S^{n-1}$ that satisfies the essential subspace concentration condition.
If $\xi$ is an essential linear subspace with respect to $\mu$ for which
$$
\mu(\xi\cap S^{n-1})=\frac{1}{n}\mu(S^{n-1})\dim\xi,
$$
then $\mu_\xi$ satisfies the essential subspace concentration condition.
\end{lemma}

For even measures, the following lemma was stated for even measures as Lemma 7.2  in \cite{BLYZ}.  However, the proof
in \cite{BLYZ} does not use the property that the measure is even.

\begin{lemma}\label{Lemma 5.2}
Let $\xi$ and $\xi'$ be complementary subspaces in $\mathbb{R}^{n}$ with $0<\dim\xi<n$. Suppose $\mu$ is a Borel measure on $S^{n-1}$ that is concentrated on $S^{n-1}\cap(\xi\cup\xi')$, and so that
$$
\mu(\xi\cap S^{n-1})=\frac{1}{n}\mu(S^{n-1})\dim\xi.
$$
If $\mu_{\xi}$ and $\mu_{\xi'}$ are cone-volume measures of convex bodies in the subspaces $\xi$ and $\xi'$, then $\mu$ is the cone-volume measure of a convex body in $\mathbb{R}^{n}$.
\end{lemma}

In addition, we also need the following lemma.

\begin{lemma}\label{Lemma 5.3}
Suppose $\mu$ is a Borel measure on $S^{n-1}$, $n\geq2$, that is not concentrated on any closed hemisphere, and $\mu$ concentrated on two complementary subspaces $\xi$ and $\xi'$ of $\mathbb{R}^{n}$.
Then, $\mu_{\xi}$ is not concentrated on any closed hemisphere of $\xi\cap S^{n-1}$ and $\mu_{\xi'}$ is not concentrated on any closed hemisphere of $\xi'\cap S^{n-1}$.
\end{lemma}
\begin{proof}
We only need prove that $\mu_{\xi}$ is not concentrated on any closed hemisphere of $\xi\cap S^{n-1}$.

Suppose $\mu_{\xi}$ is concentrated on a closed hemisphere, $C$, of $\xi\cap S^{n-1}$. Then, $\mu$ is concentrated on
$$
S^{n-1}\cap\textmd{pos}\{C\cup\xi'\}.
$$
However, $S^{n-1}\cap\textmd{pos}\{C\cup\xi'\}$ is a closed hemisphere of $S^{n-1}$. This  contradicts the conditions of the lemma. Therefore, $\mu_{\xi}$ is not concentrated on any closed hemisphere of $\xi\cap S^{n-1}$.
\end{proof}

Now, we have prepared enough to prove the main theorem of this paper.
\begin{theorem}
\label{main}
If $\mu$ is a discrete measure on $S^{n-1}$, $n\geq 1$ that is not concentrated on any closed hemisphere and satisfies the essential subspace concentration condition, then $\mu$ is the cone-volume measure of a polytope in $\mathbb{R}^{n}$.
\end{theorem}
\begin{proof}
We prove Theorem~\ref{main}  by induction on the dimension $n\geq 1$. 
If $n=1$, then the theorem trivially holds, therefore let $n\geq 2$.

If $\mu$ satisfies the strict essential subspace concentration inequality, then $\mu$ is the cone-volume measure of a polytope in $\mathbb{R}^{n}$ according to Lemma~\ref{Lemma 3.4} and Lemma~\ref{Lemma 4.4}.

Therefore we assume that  there exists an essential subspace (with respect to $\mu$), $\xi$, of $\mathbb{R}^{n}$, and a subspace, $\xi'$, of $\mathbb{R}^{n}$ such that $\xi, \xi'$ are complementary subspaces of $\mathbb{R}^{n}$, $\mu$ concentrated on $S^{n-1}\cap\{\xi\cup\xi'\}$ with
$$
\mu(S^{n-1}\cap\xi)=\frac{\dim\xi}{n}\mu(S^{n-1})\textmd{ and }\mu(S^{n-1}\cap\xi')=\frac{\dim\xi'}{n}\mu(S^{n-1}).
$$
From the fact that $\mu$ is not concentrated on a closed hemisphere and Lemma~\ref{Lemma 5.3}, we have, $\mu_{\xi}$ is not concentrated on a closed hemisphere of $S^{n-1}\cap\xi$, and $\mu_{\xi'}$ is not concentrated on a closed hemisphere of $S^{n-1}\cap\xi'$. By Lemma~\ref{Lemma 5.1}, $\mu_{\xi}$ satisfies the essential subspace concentration condition on $\xi\cap S^{n-1}$, and $\mu_{\xi'}$ satisfies the essential subspace concentration condition on $\xi'\cap S^{n-1}$. From the induction hypothesis, $\mu_{\xi}$ is the cone-volume measure of a convex body in $\xi\cap\mathbb{R}^{n}$, and $\mu_{\xi'}$ is the cone-volume measure of a convex body in $\xi'\cap\mathbb{R}^{n}$. By Lemma~\ref{Lemma 5.2}, $\mu$ is the cone-volume measure of a convex body in $\mathbb{R}^{n}$. Since $\mu$ is discrete, $\mu$ is the cone-volume measure of a polytope in $\mathbb{R}^{n}$.
\end{proof}

\section{New inequalities for cone-volume measures}
\label{secnec}

In this section, we establish some inequalities for cone-volume measures.

The following example shows that the cone-volume measure of a convex body does not need to satisfy the essential subspace concentration condition with respect to essential linear subspace.

\begin{example} 
Let $u_1,\ldots,u_n$ be an orthonormal basis of $\mathbb{R}^{n}$, and let
$W=\{x\in u_1^\bot:\,  |x\cdot u_i|\leq 1,\;i=2,\ldots,n\}$ be an $(n-1)$-dimensional
cube. For $r>0$ and $i=1,\ldots,n-1$, $\xi_i={\rm lin}\{u_1,\ldots,u_i\}$ is an essential subspace for
the cone-volume measure of the truncated pyramid $P_r=[-ru_1+rW,u_1+W]$. If $r>0$ is small, then $P_r$ approximates
$[o,u_1+W]$, and thus
$$
V_{P_r}(\xi_i\cap S^{n-1})>V_{P_r}(\{u_1\})=V([o,u_1+W])>\mbox{$\frac{i}{n}$}\, V(P_r).
$$
\end{example}

We next establish new inequalities for the cone-volume measures.
\begin{lemma}
\label{oppositefacets-est}
If $K$ is a convex body in $\mathbb{R}^{n}$, $n\geq 3$, with $o\in$Int(K), then for $u\in S^{n-1}$
\begin{equation}\tag{6.1}
V_{K}(\{u\})+V_{K}(\{-u\})+2(n-1)\sqrt{V_{K}(\{u\})V_{K}(\{-u\})}\leq V(K),
\end{equation}
with equality if and only if $F(K,-u)$ is a translate of $F(K,u)$, $K=[F(K,u), F(K,-u)]$, and $h(K,u)=h(K,-u)$.
\end{lemma}

In $\mathbb{R}^{2}$, we have

\begin{lemma}
\label{oppositefacets2}
If $K$ is a convex body containing the origin in its interior in $\mathbb{R}^{2}$, and $u\in S^1$, then
\begin{equation}\tag{6.2}
\sqrt{V_K(\{u\})}+\sqrt{V_K(\{-u\})}\leq \sqrt{V(K)},
\end{equation}
with equality if and only if $K$ is a trapezoid with two sides parallel to $u^\bot$, and $u^\bot$ contains the intersection of  the diagonals.
\end{lemma}

We obtain the following estimate from Lemma~\ref{oppositefacets-est} and Lemma~\ref{oppositefacets2}.

\begin{corollary}
\label{oppositefacetscoro}
If $K$ is a convex body in $\mathbb{R}^{n}$, $n\geq 2$ with $o\in$Int(K) and $u\in S^{n-1}$, then
\begin{equation*}
V_{K}(\{u\})\cdot V_{K}(\{-u\})\leq\frac{1}{4n^{2}}\big(V(K)\big)^{2},
\end{equation*}
with equality if and only if $F(K,-u)$ is a translate of $F(K,u)$, $K=[F(K,u), F(K,-u)]$, and $h(K,u)=h(K,-u)$.
\end{corollary}

We next prove Lemma~\ref{oppositefacets-est} and Lemma~\ref{oppositefacets2} together.
\begin{proof}
 For the case $|F(K,u)|\cdot|F(K,-u)|=0$, Lemma~\ref{oppositefacets-est} and Lemma~\ref{oppositefacets2} are trivially true. 
Thus we prove Lemma~\ref{oppositefacets-est} and Lemma~\ref{oppositefacets2}  under the condition that 
$|F(K,u)|\cdot|F(K,-u)|>0$.

Let $V_K(\{u\})=\alpha>0$ and $V_K(\{-u\})=\beta>0$, let $h_K(u)=a$ and
$h_K(-u)=b$, and for $0\leq x\leq a+b$ let
$$
K_{x}=\big((a-x)u+u^{\perp}\big)\cap K.
$$

Since $K$ is a convex body,
$$
\frac{x}{a+b}F(K,-u)+\frac{a+b-x}{a+b}F(K,u)\subset K_{x}.
$$
From this and the Brunn-Minkowski inequality,
\begin{equation}\tag{6.3}
\begin{split}
|K_{x}|&\geq\left|\frac{x}{a+b}F(K,-u)+\frac{a+b-x}{a+b}F(K,u)\right|\\
&=\left|\left(\frac{x}{a+b}F(K,-u)+\frac{a+b-x}{a+b}F(K,u)\right)_{u^{\perp}}\right|\\
&=\left|\frac{x}{a+b}F(K,-u)|_{u^{\perp}}+\frac{a+b-x}{a+b}F(K,u)|_{u^{\perp}}\right|\\
&\geq\left(\frac{x}{a+b}\big|F(K,-u)|_{u^{\perp}}\big|^{\frac{1}{n-1}}+\frac{a+b-x}{a+b}\big|F(K,u)|_{u^{\perp}}\big|^{\frac{1}{n-1}} \right)^{n-1}\\
&=\left(\frac{x}{a+b}\big|F(K,-u)\big|^{\frac{1}{n-1}}+\frac{a+b-x}{a+b}\big|F(K,u)\big|^{\frac{1}{n-1}} \right)^{n-1},
\end{split}
\end{equation}
with equality if and only if $K_{x}=\frac{x}{a+b}F(u(K,-u)+\frac{a+b-x}{a+b}F(K,u)$, and $F(K,-u)|_{u^{\perp}}$ and $F(K,u)|_{u^{\perp}}$ are homothetic.

Let $t=\frac{a+b-x}{a+b}$. From (6.3) and Fubini's formula,
\begin{equation}\tag{6.4}
\begin{split}
V(K)&=\int_{0}^{a+b}|K_{x}|dx\\
&\geq\int_{0}^{a+b}\left(\frac{x}{a+b}\big|F(K,-u)\big|^{\frac{1}{n-1}}+\frac{a+b-x}{a+b}\big|F(K,u)\big|^{\frac{1}{n-1}} \right)^{n-1}dx\\
&=(a+b)\int_0^1 \left(t|F(K,u)|^{\frac1{n-1}}+(1-t)|F(K,-u)|^{\frac1{n-1}}\right)^{n-1}\,dt\\
&=(a+b)\sum_{i=0}^{n-1}|F(K,u)|^{\frac{i}{n-1}}|F(K,-u)|^{\frac{n-1-i}{n-1}}{n-1\choose i}
\int_0^1 t^i(1-t)^{n-1-i}\,dt\\
&=\frac{a+b}n\sum_{i=0}^{n-1}|F(K,u)|^{\frac{i}{n-1}}|F(K,-u)|^{\frac{n-1-i}{n-1}}.
\end{split}
\end{equation}

Let $S_{1}=|F(K,u)|$ and $S_{2}=|F(K,-u)|$. From (6.4) and the arithmetic-geometric inequality, we have
\begin{equation}\tag{6.5}
\begin{split}
V(K)&=\frac{a+b}{n}\sum_{i=0}^{n-1}S_{1}^{\frac{i}{n-1}}S_{2}^{\frac{n-1-i}{n-1}}\\
&=\frac{a}{n}S_{1}+\frac{b}{n}S_{2}+\frac{1}{n}\sum_{i=1}^{n-1}\left(aS_{1}^{\frac{n-1-i}{n-1}}S_{2}^{\frac{i}{n-1}}+bS_{2}^{\frac{n-1-i}{n-1}}S_{1}^{\frac{i}{n-1}}  \right)\\
&\geq\alpha+\beta+2(n-1)\sqrt{\alpha\beta}.
\end{split}
\end{equation}
Thus, we get (6.1) and (6.2).

From the equality conditions for (6.3), (6.4) and the arithmetic-geometric inequality, we have, equality holds in (6.5) if and only if $F(K,u)|_{u^{\perp}}$ and $F(K,-u)|_{u^{\perp}}$ are homothetic, $K=[F(K,u), F(K,-u)]$, and
\begin{equation}\tag{6.6}
\frac{a}{b}=\left(\frac{S_{1}}{S_{2}}\right)^{\frac{2i-n+1}{n-1}},
\end{equation}
for all $1\leq i\leq n-1$.

Therefore, equality holds in (6.2) ($n=2$) if and only if $K$ is a trapezoid with two sides parallel to $u^{\perp}$, and $u^{\perp}$ contains the intersection of the diagonals.

When $n\geq3$, (6.6) hold for $i=1,..., n-1$. Thus, $\frac{a}{b}=\frac{S_{1}}{S_{2}}=1$. Therefore, equality holds in (6.1) if and only if $F(K,-u)$ is a translation of $F(K,u)$, $K=[F(K,u), F(K,-u)]$, and $h_{K}(u)=h_{K}(-u)$.
\end{proof}

\noindent{\bf Funding }
This work was supported by the Hungarian Scientific Research Fund [109789 to K.J.B., 84233 to P.H.].

\noindent{\bf Acknowledgement } We thank the unknown referees for many improvements.

\end{document}